\documentclass[a4paper,12pt]{article}
\usepackage{amsmath, amssymb, verbatim, textcomp, graphicx, framed,amsthm}
\usepackage[center]{caption}
\usepackage[position=b]{subcaption}
\usepackage{graphics}
\usepackage[a4paper,left=2cm,right=2cm,top=2cm,bottom=2cm]{geometry}
\newtheorem{theorem}{Theorem}
\newtheorem{lemma}{Lemma}
\newtheorem{proposition}{Proposition}

\newtheorem{definition}{Definition}
\newtheorem*{not*}{Notation}
\newtheorem{remark}{Remark}

\def\N{\mathbb{N}}
\def\ra{\rightarrow}
\def\R{\mathbb{R}}

\def\veps{\varepsilon}
\def\M{\mathbb E}

\begin{document}
	\title{A note on the asymptotics of random density matrices}

	\author{\Large{M. Kornyik}$^{1,2}$  \\ $^1$ E\"otv\"os Lor\'and University \\ Department of Probability Theory and 
		Statistics \\ P\'azm\'any P\'eter s\'et\'any 1/C., H-1117, Budapest, Hungary \\
		\textit{email}:koma@cs.elte.hu        \and 
		$^2$ Wigner Research Centre for Physics \\
		Quantum optics and Quantum Information Group \\
		Konkoly-Thege Mikl\'os \'ut  29-33., H-1121, Budapest, Hungary
		}

	\maketitle

\begin{abstract}

We show in this note that the asymptotic spectral distribution, location and distribution of the largest eigenvalue of a large class of random density matrices coincide with that of Wishart-type random matrices using proper scaling. As an application, we show that the asymptotic entropy production rate is logarithmic. These results are generalizations of those of Nechita, and Sommers and \. Zyczkowski.

\end{abstract}

\section{Introduction}
\label{intro}
\vspace{3mm}
Density matrices are fundamental tools of quantum mechanics and quantum information theory for describing the state of a quantum system (\cite{alicki2001quantum,nielsen2002quantum,wilde2013quantum}). While the theory of deterministic density matrices is well developed, random density matrices have not been considered by many before (\cite{Nechita,sommers2004statistical,Zyczkowski}). These matrices are particularly useful when the state of the system is either unknown or just partially known. 
Random (density) matrices also appear in tomography (\cite{GToth}), when the matrix elements come from measurements, although in this case the semi-definiteness can happen to fail.
 Due to the randomness quantities like the entropy or entanglement cannot be computed exactly, they have to be estimated, and in order to do this a probability measure has to be introduced on the set of density matrices. 
 While there is a uniquely defined uniform distribution on the set of pure states, since these are the rays of a Hilbert space, for density matrices, i.e. mixed states, there is no candidate for a canonical measure. \\
 There are two main classes of probability measures on the set of density matrices (described in more details in \cite{Zyczkowski}). The first class consists of metric measures, which are generated by metrics on the set of density matrices, e.g. the Bures distance defined by the metric $d(\rho,\sigma)=2\arccos \mathrm{Tr}(\rho^{1/2}\sigma\rho^{1/2})^{1/2}$. The second class consists of the induced measures, where density matrices are obtained by partially tracing a random pure state of a larger system. The topic of this note is confined to random density matrices of the second class. \\ In order to describe the second class assume a quantum system is in some pure state \\ $|X\rangle \in \mathcal H\otimes\mathcal K$, where $\mathcal H$ is a $p$ dimensional Hilbert space of the observer, and $\mathcal K$ is an $n$ dimensional Hilbert space representing the (unknown) environment. The state of the system in the observer's space is given by $\rho = \mathrm{Tr}_{\mathcal K} |X\rangle \langle X|$, i.e. the partial trace of $|X\rangle$ with respect to $\mathcal K$. It can  be shown that $\rho  $ has the form 
 $ XX^\dagger /\mathrm{Tr}(XX^\dagger)=XX^\dagger/ ||X||_{HS}^2$, where $X$ is a $p\times n$ matrix, $||\cdot||_{HS}$ denotes the Hilbert-Schmidt norm, and $X^\dagger$ denotes the conjugated transpose of $X$ (for more details on the tensor analytic and matrix algebraic description see \cite{joshi1995matrices}). Due to the fact that $\mathcal K$ is unknown, there is a degree of freedom in choosing the distribution of $|X\rangle$. In case the distribution of $|X\rangle$ is invariant under unitary conjugation, it can be shown that the elements of the matrix $X$ are independent and normally distributed complex random variables. By analyzing the asymptotic behavior of density matrices we can get a useful insight on large quantum systems, i.e. when $\mathcal H\otimes \mathcal K$ is of large dimension. Since the aforementioned random density matrices are functions of the generating $X$, or more precisely, functions of $XX^\dagger$, it is reasonable to analyze the spectral asymptotics of $XX^\dagger$.

\noindent The theory of the asymptotics of positive semidefinite matrices of the form $XX^\dagger$ is well established. It is known that after the proper scaling the limit of the spectral distribution is given by the compactly supported Marchenko-Pastur law and the limit distribution of the largest eigenvalue is governed by the Tracy-Widom law under quite general conditions (see eg. \cite{bai2004clt,bao,forrester1993spectrum, karoui2007tracy,mar67,nagao1995asymptotic,pastur2011eigenvalue}). Furthermore, after proper scaling the largest and smallest eigenvalues converge to the respective edge of the support with probability one. Some results from this theory can be translated to the case of random density matrices of the previously mentioned type. Nechita showed in \cite{Nechita} that after proper scaling and under the assumption that the elements of $X$ are independent with standard normal distribution the limit laws coincide with those of $XX^\dagger$.  In Section 2 we will generalize results of Nechita for a larger class of random density matrices, while Section 3 consists of the proofs of the generalized theorems. \\
Sommers and \. Zyczkowski computed the two point correlation functions of a random density matrix of the second type from the invariant ensemble. They have also shown asymptotic results for the mean of the von Neumann entropy for a special class of random density matrices (\cite{sommers2004statistical}). Section 4 generalizes their results. 
 \section{Spectral asymptotics of general random density matrices}
 First, for the sake of completeness, let us introduce some notation and definitions.
 \begin{definition}
 	The matrix $\rho\in \mathbb C^{n\times n}$ is a density matrix, if it is positive semi-definite (denoted by $\rho \geq 0$) and $\mathrm{Tr}\ \rho=1$. 
 \end{definition}
 \noindent Given any $x\in \R$, denote by $\delta_x$ the Dirac measure concentrated at $x$. 
 \begin{definition}
 	Let $c>0$ and denote by $\nu_c$ the Marchenko-Pastur law of parameter $c$, i.e. let 
 	$$ \nu_c=(1-1/c)\mathbb I_{[1,\infty)}(c) \delta_0 + \kappa_c $$
	and 
 	$$ d\kappa_c(x)= \frac{1}{2\pi c x}\sqrt{(x_+-x)(x-x_-)}\cdot \mathbb I_{[x_-,x_+]}(x) dx $$
 	with $x_\pm=(\sqrt c\pm1)^2.$ For a given set $A\subset \R$ we will denote its indicator function by $\mathbb I_A(x)$.
 \end{definition}
 \noindent Given any Hermitian matrix $A$ we will denote its $j^{th}$ largest eigenvalue by $\lambda_j(A)$. \\
 
\noindent  We note that the density matrix is in close relation with the density operator, i.e. a linear, bounded operator of a Hilbert space with trace equal to 1. It can be shown that in finite dimensional Hilbert spaces there is an equivalence between these two objects.\\

	\noindent  Now let us recall Nechita's results (\cite{Nechita}) about the asymptotics of random density matrices. 
	 Let $X=(x_{kl})_{1\leq k,l} $ be a family of independent, identically distributed (from this point on abbreviated as IID) random variables with standard complex Gaussian distribution $N_{\mathbb C}(0,1)$. Assume $p=p(n)$ is such that $\lim_{n} (p(n)/n)=c$  and consider the empirical distribution 
	\begin{equation} \label{dens_spec_emp_m} \mu_{n}= \frac1p\sum_{j=1}^p \delta_{cn\lambda_j(\rho_n)}, \end{equation}
	where $\rho_n=X_nX_n^\dagger/\mathrm{Tr}(X_nX_n^\dagger)$ and $X_n=(x_{kl})_{\substack{1\leq k \leq p(n)\\ 1\leq l \leq n}}$. Then
	$$ \mathbb P(\mu_{n} \xrightarrow[]{} \nu_c\quad \mbox{in distribution})=1, $$
	where $\nu_c$ is the Marchenko-Pastur distribution with parameter $c$ (defined rigorously in the next section).\\	Furthermore we also have
	$$ \lim_n cn\lambda_1(\rho_n) = (\sqrt c +1)^2\quad \mbox{with probability one,} $$
	and
	$$ n^{2/3} \frac{cn\lambda_1(\rho_n)- (\sqrt c+1)^2 }{ (\sqrt c+1)(1/\sqrt c +1)^{1/3}} \xrightarrow[n\to\infty]{} F_2 \quad \mbox{in distribution}, $$
	where $F_2$ denotes the Tracy-Widom distribution with parameter 2. (For more details on the Tracy-Widom law see e.g. \cite{tracy2002distribution}.)
While working with density matrices it is also interesting to analyze the asymptotic behavior of the entropy. As an application of the main results we show in the Section 4 that the asymptotic entropy rate is logarithmic.

\vspace{3mm}

\noindent In the following we will state the main results.
\begin{theorem}\label{T}
	Assume that $\{x_{ij},i,j=1,2,\ldots\}$ is a collection of complex IID random variables, with $\M [x_{kl}]=0$, $\M [x_{kl}^2]=0$ and $\M [|x_{kl}|^2]=1$, and assume that $\lim _n p(n)/n=c$. Then
	$$ \mathbb P(\mu_{n}\xrightarrow[]{} \nu_c\quad \mbox{in distribution})=1 $$
	where $\mu_{n}$ denotes the same measure as in equation (\ref{dens_spec_emp_m}) and $\M$ denotes the expectation functional.
\end{theorem}
\begin{theorem}\label{T2}
	Assume $c\in (0,\infty)$ and $\lim_n p(n)/n=c$. Consider the sequence of random density matrices $\rho_n= X_n X_n^\dagger /\mathrm{Tr}(X_nX_n^\dagger)$ (where $X_n=(x_{kl})_{\substack{1\leq k \leq p(n) \\ 1 \leq l \leq n}}$ with $(x_{kl})$ as in the previous theorem), and suppose $\M [|x_{11}|^4]<\infty$. Then $$ \lim _n cn\lambda_1(\rho_n)=( \sqrt c+1)^2\quad \text{with probability one}, $$
	$$ \lim_n cn\lambda_{n}(\rho_n)=(\sqrt c-1)^2 \quad \mbox{ with probablity one,} $$
	$$  n^{2/3}\frac{cn \lambda_1(\rho_n)-\big(1+\sqrt {\frac{p(n)}{n}} \big)^2}{\big(1+\sqrt{\frac{p(n)}{n}}\big)\big(1+1/\sqrt {\frac {p(n)}n}\big)^{1/3}} \xrightarrow[]{} F_2 \quad \mbox{in distribution.} $$
	Furthermore 
	$$ n^{2/3}\frac{cn \lambda_n(\rho_n)-\big(1-\sqrt {\frac{p(n)}{n}} \big)^2}{\big(\sqrt{\frac{p(n)}{n}}-1\big)\big(1/\sqrt {\frac {p(n)}{n}}-1\big)^{1/3}} \xrightarrow[]{} F_2 \quad \mbox{in distribution,} $$
	if $\exists\ \gamma_1>\gamma_2>0$ constants such that $\gamma_1<p(n)/n<\gamma_2$ and the $ (x_{kl}) $'s have subexponential decay, i.e. $\exists \tau_0,T>0$ independent of $k,l,n$, such that 
	$$ \mathbb P(|x_{kl}|\geq t) \leq \tau_0^{-1} \exp (-t^ {\tau_0}) \quad \text{for } t\geq T ,$$
	and $F_2$ denotes the Tracy-Widom law of parameter 2.
\end{theorem}
Note that when determining the asymptotic distribution of $\lambda_1(\rho_n)$ the quantity $\sqrt{p/n}$ cannot be replaced by $\sqrt c$ as the convergence of $p/n$ can be arbitrarily slow.
Figures (\ref{fig:TW2}) and (\ref{fig:MP}) show numerical evidence for Theorems \ref{T} and \ref{T2}. Both simulations were done using matrices with IID elements uniformly distributed on the set $\{\frac{\pm1\pm i}{\sqrt 2}\}$. The other parameters were chosen as $n=2000$, $c=1/2$ and the sample size was $5000$. The density function of the Tracy-Widom law of parameter 2 was computed with the routine {\verb dtw(x,beta=2) } of the R package called ``RMTstat", while the eigenvalue statistics were computed in Julia.
	\begin{figure}[!h]
		\centering
		\begin{subfigure}{0.48\linewidth}
			\centering
			\includegraphics[width=8.5cm,height=6.5cm]{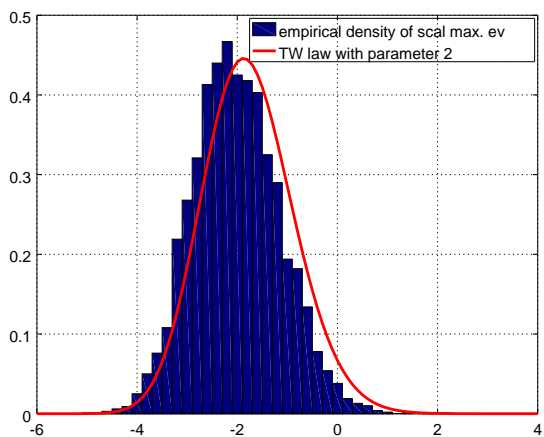}
			\caption{Tracy-Widom law and the \\ empirical density of the scaled $\lambda_1$}	\label{fig:TW2}
		\end{subfigure}  \hspace*{\fill}
		\begin{subfigure}{0.48\linewidth}
			\centering
			\includegraphics[width=8.5cm,height=6.5cm]{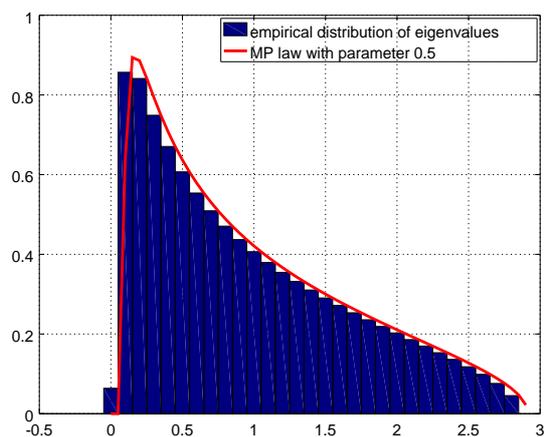}
			\caption{Marchenko-Pastur law and the \\ empirical density of the eigenvalues}
			\label{fig:MP}
		\end{subfigure}
		\caption{Numerical simulations}
		\label{fig:Num_sim}
	\end{figure}

\section{Proof of Theorem \ref{T} and Theorem \ref{T2}}
\vspace{3mm}

Before the proof let us evoke the well known theorem of Marchenko and Pastur.
\begin{theorem}[Marchenko-Pastur \cite{mar67}]
		Suppose $\{x_{kl},1\leq k,l\}$ is a family of IID complex random variables such that $\mathbb E[x_{kl}]=0$, $\mathbb E |x_{kl}|^2=1$, furthermore suppose $p=p(n)$ and $\lim_n p(n)/n=c>0$. Let $X_n:=(x_{kl})_{\substack{1\leq k \leq p(n) \\ 1\leq l \leq n}}$ and $W_n:= \frac1n X_nX_n^\dagger$ and $\mu_n':=\frac1{p(n)} \sum_{j=1}^n \delta_{\lambda_j (W_n)}$. Then we have
	\begin{equation} \label{MP_conv}\mathbb P\bigg(\mu_n' \xrightarrow[]{} \nu_c\quad \mbox{in distribution}\bigg)=1, \end{equation}
	where $\nu_c$ denotes the Marchenko-Pastur distribution with parameter $c$.
\end{theorem}
\noindent Note that this is a more general and more concisely phrased version of the original theorem. A proof of this can be found in \cite{bai2010spectral,mar67}. \\
Given a measure $\mu$ supported on $\R^+$, introduce the notation $$S(\veps,\mu):=\int\frac{1}{x+\veps} d\mu(x)\quad \veps>0.$$ 
We will also need the following lemma (for a proof see the Appendix).
\begin{lemma}[\cite{Yaskov}]\label{St-conv}
	Let $\mu,\{\mu_n,n\in \mathbb N \}$ be random probability measures with support in $[0,\infty)$. Then $\mu_n$ converges weakly to $\mu$ with probability one if and only if $S(\veps,\mu_n)$ converges to $ S(\veps,\mu)$ with probability one $\forall \veps>0$.
\end{lemma}
\noindent\begin{proof}[Proof of Theorem \ref{T}]
	According to Lemma \ref{St-conv}, equation (\ref{MP_conv}) is equivalent to \\ $\mathbb P\bigg(S(\veps,\mu_{n})\rightarrow S(\veps,\nu_c)\bigg)=1$ for all $\veps>0$, hence it is sufficient to prove that
	\begin{equation}\label{rho-St-lim} \mathbb P \big(S(\veps,\mu_{n})- S(\veps,\mu_n')\xrightarrow[]{} 0\big) =1 \quad\quad \forall \veps>0. \end{equation}
	It can be easily checked that
	\begin{equation*}\label{rho-St}S(\veps,\mu_{n})= \frac{\mathrm{Tr}\ X_nX_n^\dagger}{np(n)}\cdot S\left(\frac{\mathrm{Tr}( X_nX_n^\dagger)}{np(n)} \veps, \mu_{n,}' \right). \end{equation*}
	By assumption the elements of $X$ are IID, thus $\mathrm{Tr}\ X_nX_n^\dagger/(np(n))\xrightarrow[]{}1$ with probability one, since the Strong Law of Large Numbers (SLLN) is applicable and $\lim_n p(n)/n=c$. Also, because of $\mu$ being a finite measure, $S(\veps,\mu)$ is a continuous function of $\veps$, which implies (\ref{rho-St-lim}). 
\end{proof}

 \begin{proof}[Proof of Theorem \ref{T2}]
 	
	The first part of the theorem is quite obvious. 
	Geman showed in \cite{geman} that under the assumption of the present theorem we have
	\begin{align*} \frac{\lambda_1(X_nX_n^\dagger)}{n} &\xrightarrow[]{} (1+\sqrt c)^2 \quad \mbox{with probability one,} \\
\frac{\lambda_n(X_nX_n^\dagger)}{n}& \xrightarrow[]{} (1-\sqrt c)^2 \quad \mbox{with probability one,} \end{align*}
and	hence 
	\begin{align*} \label{r1} cn\lambda_1(\rho_n)= \frac{c n^2 }{ \mathrm{Tr}( X_nX_n^\dagger)} \frac{ \lambda_1(X_nX_n^\dagger)}{n} & \xrightarrow[]{} (1+\sqrt c)^2 \quad \mbox{with probability one, and} \\
cn\lambda_n(\rho_n)= \frac{cn^2}{\mathrm{Tr}(X_nX_n^\dagger)} \frac{\lambda_n(X_nX_n^\dagger)}{n}&\to (1-\sqrt c)^2 \quad \mbox{with probability one} \end{align*}
	according to the SLLN. \\ 
	To justify the second part we have to compare the largest eigenvalue of $X_nX^\dagger_n$ with that of $\rho_n$.  
	According to the results of Bao et al. in \cite{bao} we have
	\begin{equation}  n^ {2/3}\frac{\lambda_1(X_nX_n^\dagger)/n- \bigg(1+\sqrt \frac {p(n)} n\bigg)^2}{\bigg(1+\sqrt \frac {p(n)}n\bigg)\bigg(1+ 1/\sqrt \frac {p(n)}n \bigg)^{1/3}} \xrightarrow[]{} F_2 \quad \mbox{in distribution,} \end{equation}
	which means that it is sufficient to show that $n^{2/3}(\lambda_1(X_nX_n^\dagger)/n-cn\lambda_1( \rho_n)) \xrightarrow[]{}0$ in distribution. Since
	\begin{align*}
	\frac{\lambda_1(X_nX_n^\dagger)}n-cn\lambda_1(\rho_n) &= \frac{\lambda_1(X_nX_n^\dagger)}n\left(1- \frac{cn^2}{\mathrm{Tr}( X_nX_n^\dagger)}\right), 
	\end{align*}  
	and $\lambda_1(X_nX_n^\dagger)/n \xrightarrow[]{}(1+\sqrt c)^2$ with probability one, it remains to prove that
	\begin{align} \label{lambda_1}  n^{2/3} \left(1-  \frac{cn^2}{\mathrm{Tr}( X_nX_n^\dagger)} \right)\xrightarrow[n\rightarrow \infty]{} 0 \quad \mbox{in probability}. \end{align}
	For arbitrary, but fixed $p$ and $n$ let $S_{p,n}:=\sum_{\substack{1\leq k \leq p\\ 1 \leq l \leq n}} |x_{ij}|^2$, then $\M [S_{p,n}]=np$, $\mathrm{Var}(S_{p,n})=np(\M[ |x_{11}|^4]-1) $ and $ \frac{S_{p,n}}{np}\xrightarrow[]{}1$ with probability 1 if $p=p(n)$ and $n\to \infty$. 
	Since $\gamma_1<k(n)/n<\gamma_2$ it is enough to show that
	$$ n^{2/3}\bigg(1-\frac{np(n)}{S_{p(n),n}}\bigg) \xrightarrow[n\to\infty]{} 0 \quad \mbox{in probability.} $$
	Furthermore let $A_{p,n}:=\{S_{p,n}\in [np-\beta,np+\beta]\}$ for some small $\beta>0$ and for any fixed $\veps>0$, and let $B_{\veps,p,n}:=\{n^{2/3}\big|1-\frac{np}{S_{p,n}}\big|<\veps\}$, then
	\begin{align}
	\mathbb P(B_{\veps,p,n} )\leq  \mathbb P(B_{\veps,p,n} \cap A_{p,n})+\mathbb P(\Omega \setminus A_{p,n}),	\end{align}
	since $B_{\veps, p,n}= (B_{\veps,p,n}\cap A_{p,n}) \cup (B_{\veps,p,n}\cap(\Omega\setminus A_{p,n}))$ and $B_{\veps,p,n} \cap (\Omega\setminus A_{p,n}) \subset (\Omega\setminus A_{p,n})$. 
	First, note that 
	\begin{align*} \mathbb P\big( \Omega\setminus A_{p,n} \big)&= \mathbb P \big(\big|S_{p,n}-np\big| >\beta \big)=\mathbb P\bigg( \big|\frac{S_{p,n}}{np}-1\big|>\frac{\beta}{np}\bigg)\leq \mathbb P \bigg(\big|\frac{S_{p,n}}{np}-1\big|>0\bigg).  \end{align*}
	Fix $k$ and let $n\to \infty$, then
	$$ \limsup_n \mathbb P\big(\Omega\setminus A_{p,n}\big)\leq \limsup_n \mathbb P\bigg(\big|\frac{S_{p,n}}{np}-1\big|>0\bigg)\leq \mathbb P\bigg(\limsup_n \big|\frac{S_{p,n}}{np}-1\big|>0\bigg)=0, $$
	where in the second inequality we used Fatou's lemma and the last equality is due to the SLLN.
	Switching the role of $p$ and $n$ we obtain the same for fixed $n$ and $p\to \infty$. This proves that $\mathbb P(\Omega\setminus A_{p(n),n})\to0 $ as $n\to\infty$. \\
	On the event $A_{p,n}$ we have
	$$ n^{2/3}\frac{|S_{p,n}-np|}{S_{p,n}}\leq n^{2/3} \frac{\beta}{np-\beta}=O(n^{1/3} p^{-1}), $$
	and the quantity on the right hand side tends to $0$ whenever $p=p(n)$ and $n\to\infty$. This proves $\mathbb P(B_{\veps,p(n),n}\cap A_{p(n),n})\to 0$ as $n\to \infty$ for any fixed $\veps>0$, therefore (\ref{lambda_1}) holds true. Since convergence in probability implies convergence in distribution the proof is completed.\\
	In the case of the smallest eigenvalue Feldheim and Sodin showed in \cite{feldheim2010universality} that 
	$$n^ {2/3}\frac{\lambda_1(X_nX_n^\dagger)/n- \bigg(\sqrt {\frac {p(n)} n}-1\bigg)^2}{\bigg(\sqrt{ \frac {p(n)}n}-1\bigg)\bigg(1/\sqrt \frac {p(n)}n -1 \bigg)^{1/3}} \xrightarrow[]{} F_2 \quad \mbox{in distribution.}$$
	The proof is essentially the same as for the previous case, thus it is left to the reader.

\end{proof}

\section{Application: Asymptotic entropy}
\vspace{3mm}
In this section we are going to investigate the von Neumann entropy of random density matrices of the previously discussed type. We will prove that it exhibits a Strong Law of Large Numbers type of behavior for large systems. As it is meant to characterize the chaos present in a system, the results of this section show how much disorder is to be expected in the observed system $\mathcal H$ after tracing out the environment $\mathcal K$. We will also see that, not surprisingly, the asymptotic entropy depends on the ratio of the size of the observation space $\mathcal H$ and the environment $\mathcal K$. First let us define the von Neumann entropy of a density matrix.
\begin{definition}
	Let $\rho$ denote a density matrix on a finite dimensional Hilbert space $\mathcal H$. Denote by
	$$ H(\rho)= -\mathrm{Tr} (\rho\log \rho)= -\sum_{j} \lambda_j(\rho)\log \lambda_j(\rho) $$
	the so-called von Neumann (also known as Shannon) entropy of $\rho$. In case $0$ is an eigenvalue define $0\log 0:=0$.
\end{definition}
Sommers and \. Zyczkowski computed asymptotic results for the mean von Neumann entropy in \cite{sommers2004statistical} by showing that 
$$ \mathbb E H(\rho_n)= \ln n -\frac12 + O\bigg( \frac {\ln n} n\bigg) , $$
if $\rho_n= XX^\dagger / \mathrm{Tr}(XX^\dagger)$ is such that $X$ is an $n\times n$ Gaussian random matrix with IID elements. Our next proposition generalizes their result.
\begin{proposition}
	Let $(\rho_n)$ be a sequence of random density matrices of type introduced in Theorem \ref{T2}. Then the following relations hold with probability one:
	
	$$ \lim_n \big(H(\rho_n)-\log n\big)=\log c-\frac1{2\pi c}\int_{x_-}^{x_+} \log x\cdot \sqrt{(x_+-x)(x-x_-)}\ dx, $$
	where $x_\pm=(1\pm\sqrt c)^2$, and as a consequence
	$$ \lim_n \frac{ H(\rho_n)}{\log n}=1. $$

\end{proposition}
\begin{proof}[Proof] The proof is a series of rather simple calculations. For the sake of simplicity we will assume that $p=p(n)$ and write only $p$ throughout the proof. Using the definition of $H$ and applying algebraic transformations yield
	\begin{align}
		H(\rho_n)&= -\sum_{j=1}^p \lambda_j(\rho_n) \log \lambda_j(\rho_n)\nonumber \\
		& =-\frac{p}{cn}\cdot  \frac 1 p \sum_{j=1}^p cn\lambda_j(\rho_n) \log (cn \lambda_j(\rho_n)) + \log cn. \label{eqentr}
	\end{align}
	Notice that $\frac1{p}\sum_{j=1}^pcn\lambda_j(\rho_n) \log (cn\lambda_j(\rho_n))= \int \lambda \log\lambda\ d\mu_{n}(\lambda)$, where $\mu_n=\frac1{p}\sum_{j=1}^{p} \delta_{cn\lambda_j(\rho_n)}$, furthermore that $\eta:x\mapsto -x \log x$ with $\eta(0):=0$ is a continuous function for $x\geq 0$. Let $K>0$ and $\eta_K(x):=\max\{\eta(x),K\}$ for $x\geq0$, then 
	\begin{align}
		\left|\int \eta(x)d\mu_{n}(x) - \int \eta(x)d\nu_c(x) \right|\leq & \left| \int (\eta(x)- \eta_K(x))d\mu_{n}(x) \right| \nonumber \\ &+ \left| \int \eta_K(x)(d\mu_{n}(x)-d\nu_c(x)) \right| \nonumber\\ & +\left| \int (\eta_K(x)-\eta(x))d\nu_c(x)  \right|\nonumber\\ \label{conv_entr} &=I^{(n)}_1+I^{(n)}_2+I^{(n)}_3.
	\end{align}
	Now let $\veps>0$ be arbitrary. Since $\mu_{n}$ converges to $\nu_c$ weakly with probability one, and $\eta_K$ is a bounded, continuous function, $I^{(n)}_2<\veps/3$ if $n$ is large enough with probability one. \\
	Due to the definition of $\eta_K$ we have $\lim_{K\to \infty} \eta_K(x)=\eta(x)$ and $|\eta_K(x)|\leq |\eta(x)|$ for every $x\geq 0$. Moreover,  $\mathrm{supp}\ \nu_c$ being compact implies $\int |\eta(x)|d\nu_c(x)<\infty$. According to the dominated convergence theorem $I_3^{(n)}<\veps/3$ if $K$ is large enough. \\
	 It can be easily checked that $|\eta(x)-\eta_K(x)| \leq x^2$ for $x\geq K$ and $K\geq 1$, which means that \\ $I_1^{(n)}\leq \int_{K}^\infty x^2 d\mu_{n}(x)$. Weak convergence of $\mu_{n}$ with probability one implies $\int x^2 d\mu_{n}(x) \to \int x^2 d\nu_c $ with probability one. By writing $x^2= \min\{K, x^2\}+ (x^2 -K)\mathbb I_{(K,\infty)}(x)=f_1(x)+f_2(x) $, the function $f_1(x)$ is continuous and bounded, hence $\int f_1(x)d\mu_{n}(x) \to \int x^2 d\nu_c(x)$ with probability one if $K$ is such that $K>x_+^2$, therefore $\int f_2(x)d\mu_{n}(x)= \int_K^\infty x^2 d\mu_{n} - K \mu_{n}((K,\infty))\to 0$ with probability one. Weak convergence also implies $\mu_{n}((K,\infty))\to 0$ for any fixed $K>0$ with probability one, meaning that we have $\int_K^\infty x^2 d\mu_{n}(x)\to 0$, implying $\int|\eta(x)-\eta_K(x)|d\mu_{n}(x)\leq \int_K^\infty x^2 d\mu_{n}\to 0 $, thus $I_1^{(n)}<\veps/3$ with probability one if $n$ is large enough. \\
Summarizing the above arguments yields that the quantity in (\ref{conv_entr}) is less then $\veps$. After subtracting $\log n$ from $H( \rho_n)$ and taking the limit $n\to\infty$ we obtain
\begin{align*} \lim_n (H(\rho_n)- \log n)&= \log c -\lim_{n} \frac {p}{cn} \int x\log x\ d\mu_{n}\\ &= \log c - \int x\log x\ d\nu_c(x) \\ &= \log c - \frac{1}{2\pi c}\int_{x_-}^{x_+} \log x \sqrt{(x_+-x)(x-x_-)}\ dx \quad \mbox{ with prob. 1,} \end{align*}
due to the assumption $\lim_n p/n=c$. The second part of the proposition is a consequence of the first part and can be easily proved using equation (\ref{eqentr}).

\end{proof}
\begin{remark}
	Usually the entropy rate of a stochastic process $\{X_n,n\in\N\}$ is defined as $\lim_n \frac{H_n}{n}$, with $H_n=-\int p_n(x_1,\ldots,x_n)  \log p_n(x_1,\ldots,x_n) dx_1\ldots dx_n$ for continuous, and \\ $H_n=-\sum_{x_1,\ldots, x_n} \mathbb P(X_1=x_1,\ldots,X_n=x_n) \log \mathbb P(X_1=x_1,\ldots,X_n=x_n)$ for discrete random variables $X_1,X_2,\ldots$. \\
	In the case of this paper there is no trivial way, if any, to define a stochastic process \\ $\{X_n,n\in \N\}$ such that $H(X_1,\ldots,X_n)=H(\rho_n)$. The most natural way would be to define $(X_1,\ldots,X_n)$ so that they follow the same distribution as $(\lambda_1(\rho_n),\ldots,\lambda_n(\rho_n))$. If $F_n(x_1,\ldots,x_n)$ denotes the distribution function of $(X_1,\ldots,X_n)$ for any $n\geq 1$, then the following strong compatibility condition has to be satisfied
	$$ \forall  k\geq 0\ \ \forall x_1,\ldots,x_n \quad \int dF_{n+k}(x_1,\ldots,x_n,dx_{n+1},\ldots,dx_{n+k})=F_n(x_1,\ldots,x_n).$$
	It can be checked that this fails to happen even in the case of (Gaussian) Wishart matrices.  
	
\end{remark}
\section{Conclusion}
Nechita showed in \cite{Nechita} that the spectral asymptotics of random density matrices of the form $XX^\dagger/\mathrm{Tr}(XX^\dagger)$ coincide with that of $XX^\dagger$ after proper scaling, where the elements of $X$ are independent and their distribution is standard complex Gaussian. In this paper, the previously mentioned results are generalized for the same type of random density matrices, but for the case when $X$ comes from a larger class of random matrices.  \\
Since using the formula $XX^\dagger/\mathrm{Tr}(XX^\dagger)$ is a very simple way of simulating random density matrices, these results can be used to approximate properties like the spectral distribution, the location and distribution of the largest eigenvalue, and the von Neumann entropy of large dimensional random density matrices. \\
In the application section we have generalized results of Sommers and \. Zyczkowski by showing that random density matrices generate infinite entropy in the limit, but the production rate is logarithmic and surprisingly independent of the parameter $c$.  \\
An interesting further generalization of these results would be to consider random density matrices, where the columns of the generating $X$ matrix are independent, but the elements of a columns are not. Yaskov showed in \cite{Yaskov}, that assuming $X$ consists of independent copies of the isotropic $p$ dimensional (real) vector $\mathbf{x}_p$ the Marchenko-Pastur theorem is equivalent to a concentration of the quadratic form of the resolvent of $\frac1nXX^\dagger$. In light of the aforementioned result, it would be interesting to show whether or not the IID condition could be relaxed in Theorems \ref{T} and \ref{T2}.

\section*{Acknowledgements}
\vspace{3mm}
This work was supported by the Hungarian National Research, Development and Innovation Office (Project Nos. K124351 and 2017--1.2.1--NKP--2017--00001). \\
The author thanks O. K\'alm\'an, G. Michaletzky,  T. Kiss and T. Clark for their useful comments and observations, and G. T\'oth for bringing empirical density matrices to his attention. 

\section*{Appendix}
\vspace{3mm}
\begin{proof} [Proof of Lemma \ref{St-conv}. \cite{Yaskov}] Suppose $\mu_n$ converges to $\mu$ weakly with probability one, fix $\varepsilon>0$ and let $f(x)=\frac{1}{x+\varepsilon}$. Then $f:\R_{\geq0} \to \R_{\geq0}$ is a continuous bounded function, thus
	$$ S(\veps,\mu_n)= \int_0^\infty \frac{\mu_n(dx)}{x+\veps} \to \int_{0}^\infty \frac{\mu(dx)}{x+\veps}=S(\veps,\mu) \quad \mbox{with probability one.} $$
	Now suppose $\mathbb P (S(\veps,\mu_n)\ra S(\veps,\mu))=1$ for all $\veps>0$, then
	$$ \mathbb  P\bigg(S(\veps,\mu_n)\to S(\veps,\mu) \quad \forall\veps\in \mathbb Q\cap (0,\infty)\bigg)= 1. $$
	Let $\nu$ be an arbitrary finite measure supported on $\R_{\geq0}$ and fix $\delta>\eta>0$, then \\
	$| S(\eta,\nu)-S(\delta,\nu)|\leq |\eta-\delta|\cdot (\delta \eta)^{-1} \nu(\R_{\geq0})$. Note that $S(\veps,\mu_n)\to S(\veps,\mu)$ can be written as
	$$\forall m\in\mathbb N \ \exists N\in\mathbb N \ \mbox{such that } \forall n\geq N \ \ |S(\veps,\mu_n)-S(\veps,\mu)|<\frac1{3m} .$$
	Now for each $m$ and $\veps>0$ there is a $ q_{m,\veps}\in \mathbb Q$, such that $|q_{m,\veps}-\veps|<\frac1{3m\veps q_{m,\veps}}$, hence
	\begin{align*}
		|S(\veps,\mu_n)-S(\veps,\mu)|\leq& \ |S(\veps,\mu_n)-S(q_{m,\veps},\mu_n) |+ |S(q_{m,\veps},\mu_n)- S(q_{m,\veps},\mu) | \\ &+ |S(q_{m,\veps},\mu)-S(\veps,\mu) |<\frac1m,
	\end{align*}
	and this yields
	$$ \mathbb P (S(\veps,\mu_n)\to S(\veps,\mu) \quad \forall \veps>0)=1, $$
	implying ( Theorem 2.2 and Remark 2.3 in \cite{schilling2012bernstein} ) that $\overline \mu_n$ converges to $\overline \mu$ vaguely on every compact subset of $[0,\infty]$ with probability one. For a finite measure $\nu$ on $\R_{\geq0}$ the measure $\overline \nu$ is defined as
	$$ \overline \nu(B)=\int_0^\infty \frac{1}{x+1} \nu(dx) \quad \mbox{where B is a Borel set of }\R_{\geq0}. $$
	The function $f_z(x)=\frac{x+1}{x-z}$ with $f(\infty)=1$ is continuous on $[0,\infty]$ for all $z\in \mathbb C$ with $\mathit{Im}\ z>0$, hence
	$$ S(z,\mu_n)= \int_0^\infty \frac{1}{x-z}\mu_n(dx) = \int_0^\infty f_z(x) \mu_n(dx)\to \int_0^\infty f_z(x) \mu(dx)=S(z,\mu)  .$$
	By the standard Stieltjes continuity theorem (Theorem B.9 on page 515 in \cite{bai2010spectral}) this implies that $\mu_n$ converges to $\mu$ vaguely. For probability measures vague convergence is equivalent to weak convergence.
\end{proof}

\bibliographystyle{amsplain}
\bibliography{KM}

\end{document}